\newcommand{\lab}[1]{\label{#1}}                
 
\documentclass[12pt,reqno]{amsart}
\usepackage{amsmath}
\usepackage{amssymb}
\usepackage[left=3cm,top=3cm,right=3cm,bottom=3cm]{geometry}
\usepackage{epsfig}
\usepackage[usenames,dvipsnames]{color}

\numberwithin{equation}{section}

\newcommand{\pr}{{\bf P}}
\newcommand{\remove}[1]{}

\begin{document}
\newtheorem{theorem}{Theorem}[section]
\newtheorem{lemma}[theorem]{Lemma}
\newtheorem{sublemma}[theorem]{Sub-lemma}
\newtheorem{definition}[theorem]{Definition}
\newtheorem{conjecture}[theorem]{Conjecture}
\newtheorem{proposition}[theorem]{Proposition}
\newtheorem{claim}[theorem]{Claim}
\newtheorem{algorithm}[theorem]{Algorithm}
\newtheorem{corollary}[theorem]{Corollary}
\newtheorem{observation}[theorem]{Observation}
\newtheorem{property}[theorem]{Property}
\newtheorem{problem}[theorem]{Open Problem}
\newcommand{\R}{{\mathbb R}}
\newcommand{\N}{{\mathbb N}}
\newcommand{\Z}{{\mathbb Z}}
\newcommand\eps{\varepsilon}
\newcommand{\E}{\mathbb E}
\newcommand{\Prob}{\mathbb{P}}
\newcommand{\pl}{\textrm{C}}
\newcommand{\dang}{\textrm{dang}}
\renewcommand{\labelenumi}{(\roman{enumi})}
\newcommand{\bc}{\bar c}
\newcommand{\cal}[1]{\mathcal{#1}}
\newcommand{\G}{{\cal G}}
\newcommand{\Gnd}{\G_{n,d}}
\newcommand{\Gnp}{\G(n,p)}
\newcommand{\Hc}{{\cal H}}
\renewcommand{\P}{{\cal P}}
\newcommand{\la}{\lambda}
\newcommand{\floor}[1]{\lfloor #1 \rfloor}

\newcommand{\bel}[1]{\be\lab{#1}}
\newcommand{\ee}{\end{equation}}
\newcommand{\be}{\begin{equation}}
 \newcommand\eqn[1]{(\ref{#1})}
 \newcommand{\ex}{\E}
\newcommand{\bean}{\begin{eqnarray*}}
\newcommand{\eean}{\end{eqnarray*}}

\newcommand{\aas}{{a.a.s.}}
\newcommand{\wO}{\widetilde O}
\newcommand{\accessconst}{\gammaconst}
\newcommand{\gammaconst}{9}
\newcommand{\oldiii}{(iii)[[[*** this will be (iv) ***]]]}
\newcommand{\newiii}{[[[*** New part (iii) ***]]]}

\newcommand{\Aconst}{a} 
\newcommand{\Bconst}{b} 
\newcommand{\hatU}{\widehat U}
\newcommand{\Bin}{{\rm Bin}}
\newcommand{\tildeU}{{\widetilde U}}
\def\C{{\cal C}}
\newcommand{\proofend}{\hspace*{\fill}\mbox{$\Box$}\vspace{2ex}}

\title{Meyniel's conjecture holds for random $d$-regular graphs} 

\author{Pawe\l{} Pra\l{}at}
\address{Department of Mathematics, Ryerson University, Toronto, ON, Canada, M5B 2K3}
\thanks{The first author was supported by NSERC Discovery Grant RGPIN-2017-04402}
\email{\texttt{pralat@ryerson.ca}}

\author{Nicholas Wormald}
\thanks{The second author was supported by the Canada Research Chairs Program and NSERC, and in part by  Australian Laureate Fellowships grant FL120100125.}
\address{School of Mathematical Sciences, Monash University  VIC 3800, Australia}
\email{\tt nick.wormald@monash.edu}



\maketitle

\begin{abstract}
In the game of cops and robber, the cops try to capture a robber moving on the vertices of the graph. The minimum number of cops required to win on a given graph $G$ is called the cop number of $G$. The biggest open conjecture in this area is the one of Meyniel, which asserts that for some absolute constant $C$,  the cop number of every connected graph $G$ is at most $C \sqrt{|V(G)|}$. In a separate paper, we showed that Meyniel's conjecture holds asymptotically almost surely for the binomial random graph.  The result was obtained by showing that the conjecture holds for a general class of graphs with some specific expansion-type properties. In this paper, this deterministic result is used to show that the conjecture holds asymptotically almost surely for random $d$-regular graphs when  $d = d(n) \ge 3$.
\end{abstract}

\section{Introduction\label{intro}}

The game of \emph{Cops and Robbers}, introduced independently by Nowa\-kowski  and Winkler~\cite{nw} and Quilliot~\cite{q} more than  thirty years ago, is played on a fixed graph $G$. We will always assume that $G$ is undirected, simple, and finite.  There are two players, one being a set of $k$ \emph{cops}, where $k\ge 1$ is a fixed integer, and the other being the \emph{robber}.  The cops begin the game by occupying any set of $k$ vertices (in fact, for a connected $G$, their initial position does not matter). The robber then chooses a vertex, and the cops and robber move alternately. The players use edges to move from vertex to vertex: in a robber move, the robber alone moves, whilst in a cop move, each cop moves. More than one cop is allowed to occupy a vertex, and in a move any of the cops or robber may remain at  their current vertex. The players always know each other's current locations. If, after any move, at least one of the cops   occupies the same vertex as the robber, the game ends and the cops have won;  otherwise, that is, if the robber avoids this indefinitely, she wins. As placing a cop on each vertex guarantees that the cops win, we may define the \emph{cop number}, written $c(G)$, which is the minimum number of cops needed for the cop player to be able to force a win on $G$. The cop number was introduced by Aigner and Fromme~\cite{af}, who proved (among other things) that if $G$ is planar, then $c(G)\leq 3$. For more results on vertex pursuit games such as \emph{Cops and Robbers}, the reader is directed to the  surveys on the subject~\cite{al,ft,h} and the recent monographs~\cite{bn, book_BP}.

The most important  open problem in this area is Meyniel's conjecture (communicated by Frankl~\cite{f}). It states that $c(n) = O(\sqrt{n})$,  where $c(n)$ is the maximum of $c(G)$ over all $n$-vertex connected graphs.  If true, the estimate is best possible as one can construct a bipartite graph based on the finite projective plane with the cop number   $\Omega(\sqrt{n})$. Until recently, the best known upper bound of $O(n \log \log n / \log n)$ was given in~\cite{f}. It took 20 years to show that $c(n) = O(n/\log n)$ as proved in~\cite{eshan}. Today we know that the cop number is at most $n 2^{-(1+o(1))\sqrt{\log_2 n}}$ (which is still $n^{1-o(1)}$) for any connected graph on $n$ vertices (the result obtained independently by Lu and Peng~\cite{lp}, Scott and Sudakov~\cite{ss}, and Frieze, Krivelevich and Loh~\cite{fkl}).  Recalling the conjecture of Haj{\'o}s, which  Erd{\H o}s showed to be false for almost all graphs, to determine the truth of Meyniel's conjecture, it is natural to   check first whether random graphs provide easy counterexamples.  This paper shows that Meyniel's conjecture passes this test for random $d$-regular graphs: they satisfy the conjecture asymptotically almost surely whenever they are connected.

In a previous paper, we studied the binomial random graph $\Gnp$~\cite{PW_gnp}. Recall that this is a random graph with vertex set $[n]=\{1,2,\dots, n\}$ in which a pair of vertices appears as an edge with probability $p$, independently for each such a pair. The probability space of random $d$-regular graphs on $n$ vertices with uniform probability distribution we denote  $\mathcal{G}_{n,d}$. We say that an event in a probability space holds \emph{asymptotically almost surely} (\emph{a.a.s.}) if its probability tends to one as $n$ goes to infinity, with $n$ restricted to even integers in the case of $\mathcal{G}_{n,d}$ when $d$ is odd.

In~\cite{PW_gnp} we showed that Meyniel's conjecture holds \aas\ in $\Gnp$ provided that $np> (1/2+\eps) \log n$ for some $\eps > 0$. (Let us mention that it was shown earlier by \L{}uczak and the first author~\cite{lp2} that the cop number has a surprising ``zig-zag'' behaviour in $\Gnp$ with respect to $p$. See, for example,~\cite{PW_gnp} for more details on this model.) In this paper, we show that the conjecture holds \aas\ in $\Gnd$, provided that $d \ge 3$. Note that for $d\le 2$,   a random 2-regular graph is a.a.s.\ disconnected, and the conjecture does not apply. In any case, such graphs are trivial for current considerations.

\begin{theorem}\label{thm:main}
Suppose that $d = d(n) \ge 3$. Let $G=(V,E) \in \Gnd$. Then a.a.s.\
$$
c(G) = O(\sqrt{n}).
$$
\end{theorem}

These results for random graph models support Meyniel's conjecture, although there is currently a huge gap in the deterministic bounds: it is still not known whether there exists $\eps>0$ such that the cop number of connected graphs is $O(n^{1-\eps})$.  

We consider dense graphs in Section~\ref{s:dense} and sparse graphs in Section~\ref{s:sparse}. In each case we first recall the results from~\cite{PW_gnp} that show that the conjecture holds deterministically for a general class  of graphs  with some specific expansion-type properties. We then show that $G\in\Gnd$ is \aas\ contained in the general class. 

\section{ Proof of Theorem~\ref{thm:main}---dense case}\lab{s:dense}

In this section, we focus on dense random $d$-regular graphs, that is, graphs with average degree $d = d(n) \ge \log^4 n$. We will first cite a deterministic result that holds for a family of graphs with some specific expansion properties. After that we will show that dense random $d$-regular graphs \aas\ fall into this class of graphs and so the conjecture holds \aas\ for dense random $d$-regular graphs. We will need three different arguments for the three intervals for $d$, and we will treat them independently. Combining these observations will immediately show that Meyniel's conjecture holds for dense random $d$-regular graphs.
 
Before stating the result, we need some definitions. Let $S(v,r)$ denote the set of vertices whose distance from $v$ is precisely $r$, and $N(v,r)$ the set of vertices (``ball'') whose distance from $v$ is at most $r$. Also, $N[S]$ denotes $\bigcup_{v \in S} N(v,1)$, the closed neighbourhood of $S$,  and $N(S)=N[S] \setminus S$ denotes the (open) neighbourhood of $S$. All logarithms with no suffix are natural. 

\begin{theorem}{\bf{\cite{PW_gnp}}}\label{thm:general_dense_case}
Let $\G_n$ be a set of graphs and $d=d(n)\ge \log^3 n$. Suppose that for some positive constant $c$, for all $G_n\in\G_n$ the following properties hold.
\begin{enumerate}
\item Let $S \subseteq V(G_n)$ be any set of $s=|S|$ vertices, and let $r \in \N$. Then
$$
\left| \bigcup_{v \in S} N(v,r) \right| \ge c \min\{s d^r, n \}.
$$
Moreover, if $s$ and $r$ are such that $s d^r < n / \log n$, then
$$
\left| \bigcup_{v \in S} N(v,r) \right| \sim s d^r.
$$
\item Let $v \in V(G_n)$, and let $r \in \N$ be such that $\sqrt{n} < d^{r+1} \le \sqrt{n} \log n$. Then there exists a family 
$$
\Big\{W(u) \subseteq S(u,r+1) : u \in S(v,r) \Big\}
$$ 
of pairwise disjoint subsets such that, for each $ u \in S(v,r)$,
$$
|W(u)| \sim d^{r+1}.
$$
\end{enumerate}
Then $c(G_n) = O(\sqrt{n}).$
\end{theorem}

Note that condition (ii) holds vacuously if there is no $r \in \N$ such that $\sqrt{n} < d^{r+1} \le \sqrt{n} \log n$. More importantly, in fact, a slightly stronger result holds. Suppose that the robber plays on a graph $G=(E_1, V)$ but the cops play on a different graph, $H=(E_2, V)$, on the same vertex set. Again, the cops win if they occupy the vertex of the robber. We will use $c(G,H)$ for the counterpart of the cop number for this variant of the game. In~\cite{PW_gnp} it was shown that the same conclusion holds on this variant of the game provided that the appropriate upper and lower bounds in the hypotheses hold on the respective graphs. 

\begin{observation}[\cite{PW_gnp}]\label{obs:sandwich}
Let $\G_n$ and $\Hc_n$ be two sets of graphs and $d=d(n)\ge \log^3 n$. Suppose that for all $G_n\in\G_n$ and all $H_n\in\Hc_n$ we have the following:
\begin{enumerate}
\item for some positive constant $c$, conditions (i) and (ii) in the hypotheses of Theorem~\ref{thm:general_dense_case} are satisfied for $H_n$,
\item $G_n$ is $d$-regular.
\end{enumerate} 
Then $c(G_n, H_n) = O(\sqrt{n})$.
\end{observation}

\subsection{Interval 1: $\log^4 n \le d = d(n) < n^{1/3} / \log^3 n$}

It is known that if $d \gg \log n$ and $d\ll n^{1/3} / \log^2 n$, then there exists a coupling of $\Gnp$ with $p = \frac dn (1- O((\log n)/d)^{1/3})$, and the space $\Gnd$ of random $d$-regular graphs, such that a.a.s.\ $\Gnp$ is a subgraph of $\Gnd$~\cite{KimVu}. 

In~\cite{PW_gnp}  it was shown that the hypotheses of Theorem~\ref{thm:general_dense_case} are a.a.s.\ satisfied if we set  $\G_n=\G(n,p)$ and replace $d$ by $d_0:= p(n-1)$. Since we are only concerned with $r=O(\log n/\log \log n)$ and $d \ge \log^4 n$, 
$$
d_0^r = \big( p(n-1) \big)^r \sim d^r (1-(\log n/d)^{1/3})^r = d^r (1- O(r (\log n/d)^{1/3})) \sim d^r.
$$ 
It follows that the hypotheses of Theorem~\ref{thm:general_dense_case} are also satisfied a.a.s.\ with $\G_n=\G(n,p)$ using the ``true" value of $d$.  Hence by  Observation~\ref{obs:sandwich}, a.a.s.\ $c(G_n, H_n) = O(\sqrt{n})$ where $H_n$ is $\Gnd$ and $G_n$ is the coupled $\Gnp$ described above. With probability $1-o(1)$, $G_n\subseteq H_n$, in which case all the cops' moves are valid on the robber's graph, and thus $c(G_n, H_n) \le c(H_n)$. Theorem~\ref{thm:main} holds for the range considered in this case. 

\subsection{Interval 2: $d = d(n) \ge \sqrt{n} \log n$}

As usual, let us start with a few definitions. A \emph{dominating set} of a graph $G=(V,E)$ is a set $U \subseteq V$ such that every vertex $u \in V \setminus U$ has at least one neighbour in $U$. The \emph{domination number} of $G$, $\gamma(G)$, is the minimum cardinality of a dominating set in $G$. It is well known that for any graph $G$ on $n$ vertices with minimum degree $\delta = \delta(n) \ge 2$ we have  
$$
\gamma(G) \le \frac {1+\log(\delta+1)}{\delta+1} n.
$$
(See, for example,~\cite{AS}.) Hence, any $d$-regular graph with $d=d(n) \ge \sqrt{n} \log n$ has a dominating set of cardinality $O(n \log n / d) = O(\sqrt{n})$. Since $c(G) \le \gamma(G)$, Meyniel's conjecture holds for any such graph. Theorem~\ref{thm:main} holds for the range considered in this case. (In fact, the result is stronger as it holds deterministically.)

\subsection{Interval 3: $n^{1/3} / \log^3 n \le d = d(n) < \sqrt{n} \log n$}

We will verify that for this range of $d$, random $d$-regular graphs \aas\ satisfy the conditions (i) and (ii) in the hypotheses of Theorem~\ref{thm:general_dense_case}. In fact, only the condition (i) needs to be verified as the condition (ii) does not apply for this range of $d$, since $r\ge 1$. Moreover, as we will only need part (i) to be verified for $r = O(1)$, the following lemma will imply the result.

\begin{lemma}
Suppose that $n^{1/3} / \log^3 n \le d = d(n) < \sqrt{n} \log n$. Let $G=(V,E) \in \Gnd$. Then, there exists some positive constant $c$ such that the following properties hold \aas\
For any set $S \subseteq V$ of $s=|S| \le cn/d$ vertices 
\begin{equation}\label{eq:switching1}
\big| N[S] \big| \ge c s d.
\end{equation}
Moreover, if $s$ is such that $s d < n / \log n$, then
\begin{equation}\label{eq:switching2}
\big| N[S] \big| \sim s d.
\end{equation}
\end{lemma}

\begin{proof}
In~\cite{BS} there are bounds on the number of edges within a given set of vertices, in $\Gnd$. Some of these bounds are obtained by using switchings. However, the results obtained there do not suffice for our present needs. The main additional information we need is a bound on the number of edges between two sets of certain sizes. 

We will show the required expansion of a set $S$ to its neighbours. We do this by showing firstly that there cannot be too many edges within $S$, and secondly that there cannot be too many edges from $S$ to an ``unusually'' small set $T$, where, often, $sd$ is approximately the ``usual'' size. It follows that, if the neighbourhood of $S$ is too small, by setting it equal to $T$, we see that there is not enough room for the edges incident with $S$ (each vertex of which must have degree $d$). Hence the neighbourhood must be large.

We start with the more difficult issue: edges from $S$ to $T$.  The approach is similar to some of the results in~\cite{BS} and similar papers on random regular graphs of high degree. Suppose $|S|=s$ and $|T|=t$ where $S$ and $T$ are disjoint subsets of $V$. Let $U=V\setminus (S\cup T)$ and put $u=|U|=n-s-t$.   Moreover,   assume that $s+t\le n/3$ and $s<u-n/2$.

In $\Gnd$ consider the set of graphs $\C_i$ with exactly $i$ edges from $S$ to $T$. Since this is a uniform probability space, we may bound $\pr(\C_i)$ via the simple inequality
\bel{probbound}
\pr(\C_{i })\le \frac{ |\C_{i  }|}{|\C_{i_1}|},
\ee
which holds for any $i_1$ such that $\C_{i_1}\ne \emptyset$. We will do this for all $i>(1-\alpha)sd$, where $\alpha>0$ is sufficiently small. Let $G$ be a member of $\C_i$ where $i>0$. Consider a ``switching" applied to $G$ which is an operation consisting of the following steps: select one of the $i$ edges $ab$ in $G$ with $a\in S$ and $b\in T$, and some other edge $a'b'$ such that $aa'$ and $bb'$ are not edges of $G$ and $a'\in U$, $b'\notin S$, and replace the edges $ab$ and $a'b'$ by new edges $aa'$ and $bb'$. Call the resulting graph $G'$. Then $G'$ is clearly a $d$-regular graph and must lie in $\C_{i-1}$ since the edge $ab$ is removed and none of $a'b'$, $aa'$  or $bb'$ can join $S$ to $T$. 
There are $iud$ ways to choose one of the $i$ edges for $ab$ and one of the $ud$ pairs of vertices $a'b'$ where $a'\in U$ and $b'$ is adjacent to it. Hence,  the number $N$ of ways to choose a  graph in $\C_i$ and perform this switching is
$$
N=|\C_i| (i u d - A),
$$
where $A$ denotes the average number of choices  excluded by the other constraints. The number of exclusions due to $aa'$ or $bb'$ being an edge of $G$ is $O(id^2)$. The number excluded because $b'\in S$ is at most $isd$ since there are at most $sd$ edges from $S$ to $U$. Hence   $A \le isd + O(id^2)$, and we get $N \ge |\C_i| (i (u-s) d + O(id^2))$. 

On the other hand, the number of ways to arrive at given graph in $\C_{i-1}$ after applying such a switching to a graph in  $\C_{i}$   is at most $(sd-i+1)td$ since this is a clear upper bound on the number of choices for $aa'$ and $bb'$. Thus $N\le |\C_{i-1}| (sd-i+1)td$ and we deduce
$$
\frac{|\C_i|}{|\C_{i-1}|} \le \frac {(sd-i+1)t}{i (u-s) + O(id)}.
$$
For the present lemma, $n^{1/3+o(1)} \le d \le n^{1/2+o(1)}$. Since $u-s>n/2$, we have 
 $$
\frac{|\C_i|}{|\C_{i-1}|} \le \frac {(sd-i+1)t}{i (u-s)}(1+O(n^{-1/4})).
$$
Applying the inequality $|\C_i|/|\C_{i-1}| \le (1+O(n^{-1/4})) \beta \rho(i)$ for all $i_1 < i\le i_0$, where 
\begin{eqnarray*}
\beta &=& \frac{t}{u-s}, \\
i_0 &=& (1-\alpha)sd, \\
i_1&=& \frac{\beta}{1+\beta} sd, \\
\rho(i) &=& \frac {sd-i+1}{i}, 
\end{eqnarray*}
we obtain
\begin{eqnarray*}
\frac{|\C_{i_0}|}{|\C_{i_1}|} &\le& \Big(\beta(1+O(n^{-1/4}))\Big)^{i_0-i_1}
\rho(i_0) \rho(i_0-1) \cdots \rho(i_1+1) \\
&=& \Big(\beta(1+O(n^{-1/4}))\Big)^{i_0-i_1}
\frac{(sd-i_1) (sd-i_1-1) \cdots (sd-i_0+1) } {i_0 (i_0-1) \cdots (i_1+1) } \\
&=& \Big(\beta(1+O(n^{-1/4}))\Big)^{i_0-i_1}
\frac{(sd-i_1)!/(sd-i_0)! } {i_0!/ i_1 !} \\
&=& O(sd) \left(\beta^{1-\alpha-\beta/(1+\beta)}(1+O(n^{-1/4}))
\frac{\left(\frac{1}{1+\beta}\right)^{1/(1+\beta)}
\left(\frac{\beta}{1+\beta}\right)^{\beta/(1+\beta)}}{\alpha^\alpha(1-\alpha)^{(1-\alpha)}}  \right)^{sd},
\end{eqnarray*}
and hence by~\eqn{probbound}
$$
\pr(\C_{i_0})=O(sd) \left(\frac{\beta^{1-\alpha}(1+O(n^{-1/4}))}{(1+\beta)\alpha^\alpha(1-\alpha)^{(1-\alpha)}}  \right)^{sd}
\le \left(\frac{\beta^{1-\alpha}(1+O(n^{-1/4}))}{\alpha^\alpha(1-\alpha)^{(1-\alpha)}}  \right)^{sd}.
$$
We will apply this for $i_0 = ds - O(s \log n)$, that is, for $\alpha=O(\log n/d)$, for which 
\begin{eqnarray*}
\alpha^\alpha(1-\alpha)^{(1-\alpha)} &=& \exp \left( \alpha \log \alpha - (1+o(1)) \alpha (1-\alpha) \right) \\
&=& \exp\left( - O(d^{-1} \log^2 n) \right) = 1+O(n^{-1/4}).
\end{eqnarray*}
Hence,
$$
\pr(\C_{i_0}) \le \left( \beta^{1-\alpha}(1+O(n^{-1/4}))  \right)^{sd}.
$$
  
On the other hand, by the union bound, the probability there exists such a pair of sets $S$ and $T$ of sizes $s$ and $t = csd$, respectively, with precisely $i_0$ edges joining them is at most 
\begin{eqnarray*}
p &=& {n\choose t}{n-t \choose s}\pr(\C_{i_0}) 
\le (en/t)^{t} (en/s)^s\, \pr(\C_{i_0}) \\
&=& \left( \left( \frac {en}{csd} \right)^c \left( \frac {en}{s} \right)^{1/d} \right)^{sd}\pr(\C_{i_0}) 
= \left( \left( \frac {en}{csd} \right)^c (1+O(n^{-1/4})) \right)^{sd}\pr(\C_{i_0}).
\end{eqnarray*}
Assuming that $sd \le cn$, which implies that $t \le c^2 n$, we note that 
$$
\beta^{1-\alpha} = \left( \frac {t}{u-s} \right)^{1-\alpha} \le \left( \frac {csd}{n-c^2n-O(n/d)} \right)^{1-O(d^{-1})} = \frac {csd}{(1-c^2)n} (1+O(n^{-1/4})), 
$$
and so
\begin{eqnarray*}
p^{1/sd} &\le& \left( \frac {en}{csd} \right)^c \frac {csd}{(1-c^2)n} (1+O(n^{-1/4})) \\
&\le& \frac {c e^c}{(1-c^2) c^c} \left (\frac {sd}{n} \right)^{1-c} (1+O(n^{-1/4})) \\
&\le& \frac {c^2 e^c}{(1-c^2) c^{2c}}  (1+O(n^{-1/4})) \le 1/2
\end{eqnarray*}
for $c$ small enough. Finally, we take the union bound over all $s$ such that $n / \log n \le sd \le cn$ and all $i_0 = ds - O(s \log n)$, to conclude that~(\ref{eq:switching1}) holds \aas\ for this range of $s$; note that it will follow immediately from~(\ref{eq:switching2}) that~(\ref{eq:switching1}) holds for the remaining range of $s$.

Let $\eps = 1/\log \log \log n$. By a similar argument, the probability there exists such a pair of sets $S$ and $T$ of sizes $s$ and $t = sd(1-\eps)$, respectively, with precisely $i_0$ edges joining them is at most 
$$
p = \left( \left( \frac {en}{sd(1-\eps)} \right)^{1-\eps} (1+O(n^{-1/4})) \right)^{sd}\pr(\C_{i_0}).
$$
This time we assume that $sd \le n/\log n$ which implies that $t = O(n/\log n)$. It follows that 
$$
\beta^{1-\alpha} = \left( \frac {t}{u-s} \right)^{1-\alpha} = \left( \frac {sd(1-\eps)}{n-O(n/\log n)} \right)^{1-O(d^{-1})} = \frac {sd}{n} (1-\eps)(1+O(1/\log n)), 
$$
and so
\begin{eqnarray*}
p^{1/sd} &\le& \left( \frac {en}{sd(1-\eps)} \right)^{1-\eps} \frac {sd}{n} (1-\eps)(1+O(1/\log n)) \\
&=& \left( \frac {e}{1-\eps} \right)^{1-\eps} (1-\eps) \left(\frac {n}{sd} \right)^{-\eps} (1+O(1/\log n)) \\
&\sim& \exp \big( 1 -\eps \log (n/sd) \big) \\
&=& \exp \big( - \Omega( \eps \log \log n ) \big) = o(1).
\end{eqnarray*}
As before, we take the union bound over all $s$ such that $sd \le n / \log n$ and all $i_0 = ds - O(s \log n)$, to conclude that~(\ref{eq:switching2}) holds \aas\ (We have ignored the fact that every vertex of $T$ must be adjacent to a vertex of $S$. Using this would improve the bound, but what we already have suffices.)

\smallskip

It remains to show that \aas\ each set of $s\le cn/d$ vertices induces at most $s\log n$ edges. This can be shown using an approach similar to the above argument for edges out of $S$, but is significantly simpler. Hence, we only outline the argument here. Fix $S$ of size $s$  and consider the set of graphs $\C_i$ with exactly $i$ edges in the graph induced by $S$. This time we get
$$
\frac{|\C_i|}{|\C_{i-1}|} \le \frac {{sd-2(i-1) \choose 2}}{i (n-2s)d + O(id^2)} = \frac {(sd-2i+2)(sd-2i+1)}{2ind} (1+O(n^{-1/4})).
$$
For each $i \ge s \log n$ we have
$$
\pr(\C_{i}) \le \frac{ |\C_{i}|}{|\C_{0}|} \le \left( \frac {1}{nd} \right)^i \frac {(sd)! / (sd-2i)!}{i!} \le \left( \frac {1}{nd} \right)^i  \frac {(sd)^{2i}}{(i/e)^i} \le \left( \frac {es^2d}{ni} \right)^i \le \left( \frac {esd}{n \log n} \right)^{s \log n} 
$$
and so the probability that there exits set $S$ of size $s$ that induces $i$ edges  ($s \log n \le i \le {s \choose 2}$)  is at most
\begin{eqnarray*}
{n \choose s}  \sum_{i = s \log n}^{s \choose 2}  \pr(\C_{i}) &=& O(s^2)   \left( \left( \frac {esd}{n \log n} \right)^{\log n} \left( \frac {en}{s} \right) \right)^s \\
&=&   O(s^2)   \left( \exp \left( -\Omega( (\log n)(\log \log n) ) + O(\log n)  \right) \right)^s \le n^{-s}.
\end{eqnarray*}
Since $\sum_{s=1}^n n^{-s} = O(n^{-1}) = o(1)$, \aas\ no set $S$ has this property.
\end{proof}

Finally, we note that extension of the coupling between random graphs and random regular graphs to other ranges of $d$ would permit the argument for interval 1 to be applied. A recent paper of  Dudek,  Frieze,  Ruci\'nski and {\v S}iliekis~\cite{Sandwich} does such an extension that would apply to intervals 2 and 3. However, the approach we used for these intervals is not so complicated, much shorter than proving the coupling. Use of the coupling for interval 1 was of a greater benefit. No such coupling can possibly cover the sparse case in the following section.

\section{ Proof of Theorem~\ref{thm:main}---sparse case}\lab{s:sparse}

In this section we treat the sparse case, that is, when $3 \le d = d(n) < \log^{4} n$. As in the dense case, we will first cite a deterministic result that holds for a family of graphs with some specific expansion properties. After that, we will show that sparse random $d$-regular graphs \aas\ fall into this class of graphs and so the conjecture holds \aas\ for sparse random $d$-regular graphs. Before stating the result, we need some definitions. 

We define $S(V',r)$ to be the set of vertices whose distance to $V'$ is exactly $r$, and  $N(V',r)$ the set of vertices of distance at most $r$ from $V'$. A subset $U$ of $V(G)$ is {\em $(t,c_1,c_2)$-accessible} if we can choose a family $\{W(w):w\in U\}$ of pairwise disjoint subsets of $V(G)$ such that $W(w)\subseteq N(w,t)$ for each $w$, and
$$
|W(w)| \ge c_1\min\left\{  d^{t},\frac{c_2n}{  |U|}\right\}.
$$
This definition will be used for constants $c_1$ and $c_2$, and large $t$. The motivation is that, for an accessible set $U$, there are ``large'' sets of vertices $W(w)$ which are disjoint for each $w\in U$, such that any cop in $W(w)$ can reach  $w$ within $t$ steps. The following result comes from~\cite[Theorem 4.1]{PW_gnp} upon setting $X(G_n)=\emptyset$.

\begin{theorem} \lab{t:main}
Let $\G_n$ be a set of graphs and $d=d(n)\ge 2$. Suppose that $d < \log^J n$ for some fixed $J$ and that for some positive constants $\delta$ and  $a_i$ ($1\le i\le 5$), for all $G_n\in\G_n$ the following hold.
\begin{enumerate}
\item For all $v \in V(G_n)$, all $r \ge 1$ with $d^r < n^{1/2+\delta}$, all $r'$ that satisfy the same constraints as $r$, and all $V' \subseteq N(v,r)$ with $|V'|=k$ such that $k d^{r'}\le n/\log^J n$, we have  
$$
a_1k d^{r'}\le  |S(V',r')|\le   a_2k d^{r'}.
$$
In particular, with $k=1$
$$
a_1 d^{r'} \le |S(v,r')| \le a_2 d^{r'}.
$$
\item Let $r$ satisfy $   n^{1/4-\delta} < (d+1)d^r<n^{1/4+\delta}  $, let $r'$ satisfy the same constraints as $r$. Let $v\in V(G_n)$, $A\subseteq S(v,r)$, and $U=\bigcup_{a\in A} S(a,r')$ with  $|A|> n^{1/4-\delta}$ and $ d^{r+r'}< a_3n/ |U|$.  Then there exists a set $Q$ such that $|S(a,r')\cap Q|<n^{1/4-2\delta}$ for all $a\in A$, and such that $U\setminus Q$ is  $(r+r'+1,a_4,a_5)$-accessible.  
\item $G_n$ is connected.
\end{enumerate}
Then $c(G_n) = O(\sqrt{n})$. 
\end{theorem}

Suppose that $3 \le d=d(n) < \log^{4} n$.  We will verify that $G  \in \Gnd$  \aas\ satisfies the conditions in the hypotheses of Theorem~\ref{t:main} and, as a result,  Theorem~\ref{thm:main}  holds for such $d$. 

For convenience when comparing with $\mathcal{G}(n,p)$, we consider $(d+1)$-regular graphs, starting with   some typical properties of $\mathcal{G}_{n,d+1}$. We will use the fact that the number of vertices in a balanced tree of height $r$ and all internal vertices of degree $d+1\ge 3$ is
\bel{treesize}
1+  \sum_{j=0}^{r-1}(d+1)d^j =\Theta\Big(\sum_{j=-\infty}^{r-1}(d+1)d^j\Big) =  \Theta\big((d+1)d^r/(d-1)\big) = \Theta(d^r).
\ee

Our analysis of $\G_{n,d+1}$ exploits its relationship to the pairing model $\P_{n,d+1}$   that we define next. Suppose that $dn$ is even, as in the case of random regular graphs, and consider $dn$ points partitioned into $n$ labelled buckets $v_1,v_2,\ldots,v_n$ of $d$ points each. A \textit{pairing} of these points is a perfect matching into $dn/2$ pairs. Given a pairing $P$, we may construct a multigraph $\P_{n,d}$, with loops allowed, as follows: the vertices are the buckets $v_1,v_2,\ldots, v_n$, and a pair $\{x,y\}$ in $P$ corresponds to an edge $v_iv_j$ in $\P_{n,d+1}$ if $x$ and $y$ are contained in the buckets $v_i$ and $v_j$, respectively. It is an easy fact that the probability of a random pairing corresponding to a given simple graph $G$ is independent of the graph, hence the restriction of the probability space of random pairings to simple graphs is precisely $\G_{n,d}$. 
 One of the advantages of using this model is that the pairs may be chosen sequentially, at each step choosing a point using any rule (possibly randomised) that depends only on the pairs so far chosen (such as the least-numbered point not yet paired, under some numbering scheme), and pairing it with a point chosen uniformly at random over the remaining (unchosen) points. For more information on this model, see, for example,~\cite{models}.  

To prove the desired expansion property for $\G_{n,d+1}$, we need a separate argument for small sets, where the probability of failing to expand can be much larger than the probability that the pairing model produces loops and multiple edges.

In the following lemma, we define the {\em excess} of a graph $H$ to be the number of edges it has in excess of that of a tree on the same number of vertices, i.e.\ $|E(H)|-|V(H)|+1$. 

\begin{lemma}\lab{l:basics}
Let $2\le d< \log^4 n $, and consider $G\in\G_{n,d+1}$. 
\begin{enumerate} 
\item Let $\eps>0$ and fix $K\ge 10$.   Suppose that $k,\, r\ge 1$ satisfy  $k(d+1)  d^{r-1}  < \log^{K} n$ and let $V'$ be a set of vertices of cardinality $k$.  Let $H(V',r)$ denote the graph induced by all the edges with at least one vertex in $N(V',r-1)$. Then  the probability that  $H(V',r)$ has excess at least $k(1+\eps)$ is bounded above by a function $p=o( n^{-k- \eps/2})$, where the convergence is uniform over all relevant $k$, $r$ and $V'$.

\item Let $\eta>0$,  $0<\delta<1/2$ be fixed. Then \aas\   for all $k=O(n^{1/2+\delta})$, all $V'\subseteq V(G)$ with $|V'|=k$ and all $r\ge 1$ with $k(d+1) d^{r-1}\le n/\log n$, we have  $|S(V',r)|\ge (1/9-\eta)k(d+1)d^{r-1}$. 

\item   Let $\delta$ be fixed with  $0<\delta<1/32$.  Then  \aas\  either $G$ is disconnected, or the following is true. Let $r$ satisfy $   n^{1/4-\delta} < (d+1)d^r<n^{1/4+\delta}  $, let $r'$ satisfy the same constraints as $r$, let   $v\in V(G)$, $A\subseteq S(v,r)$, and  $U=\bigcup_{u\in A} S(u,r')$ with  $|A|> n^{1/4-\delta}$ and $ d^{r+r'}< n/\gammaconst|U|$.  Then there exists a set $Q$ of cardinality at most $  n^{5\delta} $ such that $U\setminus Q$ is  $(r+r'+1,2/5,1/\accessconst)$-accessible.
\end{enumerate}
\end{lemma}

\bigskip

\begin{proof}[Proof of Theorem~\ref{thm:main} for the sparse case]

We will use Lemma~\ref{l:basics} to show that $G_n\in\G_{n,d+1}$ \aas\ satisfies the conditions in the first two hypotheses of Theorem~\ref{t:main}, with a suitable choice of the constants.

The upper bound in Theorem~\ref{t:main}(i) follows with $a_2=1$ immediately from the regularity of $G_n$. Hence,  $V' \subseteq N(v,r)$ implies $k=O(n^{1/2+\delta})$, and so the lower bound in Theorem~\ref{t:main}(i) follows from the lower bound given a.a.s.\ in Lemma~\ref{l:basics}(ii), applied with $r$ replaced by $r'$.

The condition in Theorem~\ref{t:main}(ii) follows directly from the property in Lemma~\ref{l:basics}(iii) with any $0<\delta<1/32$, $a_3=1/9$, $a_4=2/5$ and $a_5=1/\accessconst$. (The bound $|Q|\le n^{5\delta}$ immediately implies the required upper bound on $ |S(a,r')\cap Q|$.)

Condition (iii) follows from the fact that a random $d$-regular graph is a.a.s.\ connected for any $d\ge 3$ (see Bollob{\'a}s~\cite{BB} and Wormald~\cite{models} for $d$ fixed; see Cooper, Frieze and Reed~\cite{CFR}, and Krivelevich, Sudakov,  Vu and Wormald~\cite{KSVW} for large $d$). The theorem follows.
\end{proof}

The rest of the paper is devoted to proving Lemma~\ref{l:basics}.
To assist the reader, we first explain something of the difficulties which will be encountered when trying to establish $t$-accessibility in part (iii) of the lemma.  Our general approach is to explore the graph in successive neighbourhoods away from the vertex $v$ and try to find $W(w)$  in part of the neighbourhood ``past'' $w$. One of the awkward situations occurs for instance when $d+1=3$, and   one candidate $w$ has no neighbour in the direction   away from $v$.  The probability this happens to $w$ when about $\sqrt n$ vertices have already been explored is about $1/n$. So, summing over all such $w$, it can happen to  vertex $v$ with probability approximately $1/\sqrt n$. To cope with this and other cases where the neighbourhood of $w$ is not suitable, we will place such $w$  into the set $Q$ of vertices that do not generate $W$'s.

\begin{proof}[Proof of  Lemma~\ref{l:basics}(i)]  

Let $V'\subseteq V(G)$, and for $i\ge 0$ for brevity let $L_i$ denote $S(V',i)$, i.e.\ the set of vertices at distance $i$ from $V'$. Let $k=|V'|$ and let $r  $ be maximal such that $k(d+1)  d^{r -1} < \log^{K} n$.  
 
We may assume $r\ge 1$ since otherwise there is nothing to prove.   Expose the neighbourhood of $V'$ in  $\P_{n,d+1}$ in breadth-first search style out to distance $r$, in the following manner. Initially, all vertices in $L_0=V'$ are called \emph{exposed}, but not the points in them. During the process, some points will become exposed, and any vertex containing an exposed point will automatically be called exposed. There are $r$ rounds ($i=0,1,\ldots, r-1$), and in round $i$,   the unexposed points in  vertices in $L_i$ are exposed one by one. To expose a point, say $u$, in $L_i$, requires choosing the second point of its pair, say $u'$, uniformly at random from the remaining unexposed points; $u'$ is then also called exposed. (See the description above of how we can choose the pairs in the pairing model.) The vertex containing $u'$ must be in $L_{i+1}$ if it was not already exposed  at the start of that round (that is, if it does not belong to $L_i$).    For our initial examination, we terminate this exposure process after the completion of round $r-1$. 

Note that by the $(d+1)$-regularity, $|L_{r-1}|\le k(d+1)d^{r-2}$, and of course each vertex in $L_{r-1}$ contains at most $d+1$ points. So, by the upper bound on $r$, at the time this exposure process terminates,   $O(\log^{K}n) $ points, or vertices, are exposed, and hence $O(d\log^{K}n)$ points  (both exposed and not)  lie in exposed vertices. Let us call a pair {\em bad} if, at the time when the second point of the pair is exposed, that point's vertex is already exposed. Then the probability that a pair is bad is $O( (dn)^{-1}d\log^{K }n)=O(n^{-1} \log^{K }n)$, since at each step there are asymptotically $(d+1)n$ points remaining unexposed.   This is true conditional upon the history of the process. We now bound the number of bad pairs using stochastic domination by a binomial random variable. This is a standard idea (see~\cite[Section 10.5]{AS} for a similar example): we can in advance flip a sufficient number of independent coins whose  probability of ``heads'' is equal to an upper bound on the conditional probability of a bad edge at every step. Letting  $B_i$  denote the event that there are at least $i$ bad pairs exposed in the process started at $V'$, we conclude that   $\pr(B_i)=O( j^in^{-i}\log^{Ki}n)$, where $j = O(d\log^{K}n) $ is the number of coins used, an upper bound on the number of pairs exposed, and $j^i$ bounds the (at most)  $j\choose i$ choices for which of the coins are bad. Hence $\pr(B_i)= \big(\wO(n^{-1 })\big)^i$, considering the logarithmic  upper bound on $d$. (We use $\wO(f(n))$ as usual to denote $ O(f(n))\log^{O(1)} n$.)
 
After the exposure process, all points in $L_0, L_1, \ldots , L_{r-1}$ have been exposed. Note that no pair between two vertices in $L_{r}$ (or from those vertices to the rest of the graph) has been exposed.  Let ${\cal S}_0$ be the event that the exposure process has not revealed a loop or multiple edge. Then $\pr(\overline{\cal S}_0)\le\pr(B_1)=o(1)$ and thus
\bel{prS0}
\pr(  {\cal S}_0)\sim 1.
\ee
Next, let us continue to generate the rest of the random pairing, after completion of the above exposure process. Let ${\cal S}_1$ be the event that the rest of the pairing reveals no loops or multiple edges. No pair in the rest of the pairing can create a multiple edge with what is already exposed, so ${\cal S}_1$ is the event that a random pairing on the ``remaining" degree sequence creates a simple graph. This depends on the degree sequence of the graph induced by the remaining pairing, but there are certainly   $n-\wO(1)$ vertices of degree $d+1$ and $\wO(1)$ of smaller degree. Let $P_1$ denote the set of pairs exposed in the exposure process. Recalling again that $d=\wO(1)$, McKay's formula~\cite{M} (see also the discussion near~\cite[(5)]{models}) shows that
\bel{sbound}
\pr({\cal S}_1\mid P_1)=\exp\big( -d/2-d^2/4 +\wO(1/n)\big),
\ee
where the bounds implicit in $\wO()$ are uniform over all possible pairings $P_1$. Hence, $\pr({\cal S}_1\mid {\cal S}_0)\sim\exp( - d/2-d^2/4)$ with the same uniformity, which implies using~\eqn{prS0} that 
 $$
 \pr( {\cal S}_0\cap {\cal S}_1) \sim\exp( - d/2-d^2/4).
 $$
Also
\bean
\pr({\cal S}_1\cap{\cal S}_0\cap B_i)&=& \pr( {\cal S}_0\cap B_i)\pr({\cal S}_1\mid {\cal S}_0\cap B_i)\\
&\le &  \pr( B_i)\pr({\cal S}_1\mid {\cal S}_0\cap B_i)\\
&=&\big(\wO(n^{-1 })\big)^i\exp( - d/2-d^2/4)
\eean
by the bound above for $\pr( B_i)$, and also~\eqn{sbound} once again. These two conclusions give $\pr( B_i\mid {\cal S}_0\cap{\cal S}_1)=\big(\wO(n^{-1 })\big)^i$. The excess of $H(V',r)$ is just the number of bad pairs in the exposure process, and (i) follows upon setting $i=\lceil k(1+\eps)\rceil$, as
$$
\big(\wO(n^{-1 })\big)^{k(1+\eps)} = \big(o(n^{-1 })\big)^{k(1+\eps/2)} = o \big(n^{-k(1+\eps/2)}\big) = o \big(n^{-k-\eps/2}\big). 
$$
\proofend 

\bigskip

\noindent
{\em Proof of  Lemma~\ref{l:basics}}(ii).  
Choose relevant values of $k$ and $r$, and $V'$ of cardinality $k$. We will show that the required expansion occurs except for an event with probability so small that taking a union bound over all $k$ and $r$ and all ${n\choose k}$ sets $V'$ still gives $o(1)$.

Set $K=17$ and $0<\eps<1/3$. Firstly, in the case that $k(d+1)  d^{r-1}  < \log^{K } n$,  using the union bound over all $n\choose k$ choices for $V'$ and  the $O(\log^K n)$ permissible values of $k$ and $O(\log \log n)$ values of $r$, we have by (i) that \aas\ the excess of $H(V',r)$ is at most $k(1+\eps)$ in all such cases. Subject to this condition, the minimum number of vertices in $S(V',r)$ is achieved when each excess edge  joins two vertices in $V'$ and the rest of the neighbourhoods expand  in a ``$d$-regular'' fashion. (Even if this is infeasible, the same final bound will apply.) Then
$$
1-(2+2\eps)/(d+1) \le \frac {|S(V',r)|}{k(d+1)d^{r-1}} \le 1. 
$$
(Note that edges between two vertices in $S(V',r)$ are irrelevant.) Since $d\ge 2$ and $\eps<1/3$, this give the bound required in this case. As an aside, if $|V'|=1$ the conclusion can be strengthened slightly, since then, as $G$ is a graph, no edge can join two vertices in $V'$. (This makes the most difference when $d=2$; our overall argument will suffice even without this strengthening.)

So, from now on, we can assume 
\bel{irass}
k(d+1)  d^{r-1}  \ge \log^{K } n.
\ee  Let $r_0 $ be the minimal integer such that both $r_0\ge 1$ and $k(d+1)  d^{r_0-1}  \ge \log^{K } n$. There will essentially be two cases. If $r_0=1$, we will start the exposure process from the set $V'$ in  $\P_{n,d+1}$, much as we did in the proof of part (i), with initially $k(d+1)$ unexposed points.  However, if $r_0\ge 2$, we need more preparation. In this case  $r_0-1$ is small enough, by minimality of $r_0$, for (i) to apply with $r$ temporarily reset equal to $r_0-1$. Applying (i) as in the previous paragraph, we have that with probability  $1-o(n^{-k-\eps/2})$ the excess of $H(V',r_0-1)$ is at most $k(1+\eps)$. We next condition on the graph $H(V',r_0-1)$, for any such graph included in this event (i.e.\ with excess at most $k(1+\eps)$). Let $V_0=V_0(V')$ denote the set of  vertices of distance at most $r_0-2$ from $V'$. The graph, $G'$ say, induced by $V(G)-V_0$ (whose edges are precisely those  edges of $G$ that are not edges of $H(V',r_0-1)$) then has its degree sequence determined, and it will occur u.a.r.\ given this degree sequence. We analyse this random graph $G'$ again using the pairing model, implementing the exposure process starting at the vertices in $S(V',r_0-1)$. All vertices except these ones have degree $d+1$ in $G'$, and since the excess of $H(V',r_0-1)$ is at most $k(1+\eps)$, we may assume that the total degree of these vertices in $G'$   is at least $k(d-1-2\eps)  d^{r_0-1} $, since the minimum for a given excess would be attained if all excess edges joined vertices in $V'$. Hence, the number $x$ of points in these vertices in the pairing model for $G'$  satisfies
\bel{xbound}
x\ge k(d-1-2\eps)  d^{r_0-1}.
\ee
For $r_0\ge 2$, this gives a lower bound on the number of unexposed points in  the vertices at distance $r_0-1$ from $V'$, whilst for $r_0=1$, the number is 
  \bel{x0}
x=k(d+1)
\ee
as shown above. Beginning with these vertices, for $r\ge r_0$ we expose the successive sets $S(V',r)$ in the pairing (this round, when points in $S(V',r-1)$ are being exposed, we call round $r$ of the exposure process), for increasing $r$ up until the end of the last round $r$ for which the upper bound $n/\log n$ on  $k(d+1)d^{r-1}$ in the lemma statement (ii) is still valid. The  number $N_r$ of points which have been exposed in total at the end of round $r$ is   random but certainly we have $|S(V',r)|=O(kd^r)$ and $N_r=O(kd^{r })$ in view of~\eqn{treesize}. By our termination condition, this  implies  $N_r=O(n /\log n)$.

Since every exposed vertex contains an exposed point after the end of any round, the number of exposed vertices during the round is at most $N_r$, and the number of unexposed points in exposed vertices is at most $dN_r$ at all times during the round. Thus, for each pair exposed during round $r$, the probability it is bad (i.e.\ joins two exposed vertices), as defined in (i), is at most $d N_r/\big((d+1)n -dN_r\big)<\gamma$ for some $\gamma = O(N_r/n) = O(1/\log n)$. This, as in the argument  in (i), is true conditional on the history of the process. Beginning with $r=r_0$, consider round $r$ for general $r$. (Note that the rounds are numbered according to the value of $r$, so the first round is always round $r_0$.) We expose the unexposed points that are sitting in vertices in $S(V',r-1)$ sequentially, and as before bound the number of bad edges  stochastically   by a binomial random variable represented by coin flips. This time, a bad flip  occurs with probability $\gamma$. Let $R_{i,\beta }$ be the event that in the first $i$ coins   there are  at least $\beta i$ bad flips. Then 
$$
\pr(R_{i,\beta}) \le {i\choose \beta i}\gamma^{\beta i}\le   (e  \gamma/\beta )^{\beta i}.
$$
We wish to choose   values   $\beta_{r}$ such that if we set $\beta =  \beta_{r}$ and  $i=i_r:=k(d+1)d^{r-1}$ (which is an upper bound on the number of pairs exposed in round $r$),  we obtain $\pr(R_{i,\beta}){n \choose k} = o(e^{-d^2}/n^2)$ for all $r\ge r_0$ up to the  upper bound $r=O(\log_d n)$ implied by the hypotheses of (ii). (We suppress the dependence on $k$ in the notation $\beta_r$ and $i_r$.) It then follows by the union bound that, with probability $1-o(e^{-d^2})$, for all $k$ ($=O(n^{1/2+\delta})$) and all sets $V'$ of cardinality $k$, none of the events $R_{i,\beta}$ holds. Considering the lower bound implied by~\eqn{sbound} on the probability that   $\P_{n,d+1}$ gives a simple graph (this works as in (i) in the case $r_0\ge 2$, and applies to a perfectly $(d+1)$-regular graph otherwise), we immediately conclude that \aas\ none of the events corresponding to the $R_{i,\beta}$ hold for $\G_{n,d}$.  That is, for all starting sets $V'$, each round $r$ has less than $\beta i$ bad edges in the ``large neighbourhoods'' part of the neighbourhood growth. We will then compute the implications of this property for expansion. 

Using the familiar  bound on ${n \choose k}$ and recalling $\gamma = O(i_r/n)$, the desired bound on $\pr(R_{i,\beta})$ will hold provided that
\bel{wanted}
\left(\frac{O(i_r)}{\beta_r n}\right)^{\beta_r i_r}\left(\frac{en}{k}\right)^k=o(n^{-2}e^{-d^2}).
\ee
For $0<\eps<1/3$ as before, define
$$
\beta_r=\left\{
\begin{array}{r l}
(1+\eps)k/i_r  & \  \mbox{if }i_r\le k\log^{3}n,\\
 \rule{0cm}{0.5cm}1/(\log \log n)^{ 2},& \  \mbox{if }i_r\ge n/\log^{10}n,\\
   \rule{0cm}{0.5cm}(1+\eps)/\log^2 n,   & \  \mbox{otherwise }.
 \end{array}
\right.
$$
We next show that~\eqn{wanted} holds for each of the three ranges in the definition of $\beta_r$.

\medskip

\noindent
{\bf Case 1: } $i_r\ge n/\log^{10}n$.
\smallskip

Here in the left hand side of~\eqn{wanted} we have $ O(i_r)/\beta_r n =  O(1/\sqrt{\log n})$ in view of the upper bound
\bel{maxir}
i_r\le n/\log n
\ee  
given in the hypothesis of (ii).  (Note our convention $a/bc=a/(bc)$, here and throughout the rest of the paper.)   The maximum value of $(en/k)^k$ occurs when $k$ is maximised, at $k=O(n^{1/2+\delta})$. Recalling $d<\log^4 n$,~\eqn{wanted} follows easily here.
\medskip

\noindent
{\bf Case 2: } $ k\log^{3}n< i_r< n/\log^{10}n$.
\smallskip

Here in the left hand side of~\eqn{wanted} we have $(O(i_r)/\beta_r n)^{\beta_r i_r}  =  (O(1)/\log^8  n)^{  i_r/\log^{2}n}$. In this case,  $i_r/\log^{2}n  >  k\log n$ and so $(O(1)/\log  n)^{4i_r/\log^{2}n}=o(n^{-k-2})$.
Also by~\eqn{irass}, $i_r> \log^Kn = \log^{17 }n$ whereas $d^2<\log^8 n$, and so $(O(1)/\log  n)^{4i_r/\log^{2}n}=o(e^{-d^2})$. (Actually here we only needed $K\ge 10$.) Thus~\eqn{wanted} follows in this case.
\smallskip

\noindent
{\bf Case 3: } $ i_r\le k\log^{3}n$. 
\smallskip

Note that~\eqn{irass} ensures that $i_r\ge \log^K n$ and hence $k\ge \log^{K-3}n=\log^{14} n$.
Here we have $i_r=\wO(k)$ and $1/\beta_r=O(i_r/k)=\wO(1) $. Thus in the left hand side of~\eqn{wanted}  we have  $(O(i_r)/\beta_r n)^{\beta_r i_r}  =  (\wO(k)/n)^{ (1+\eps)k}=  o(1/n^2)( k/n)^{ (1+\eps/2)k}$. Now~\eqn{wanted} easily follows since $d<\log^4 n<\sqrt k$ and $k=O(n^{1/2+\delta})$. (Note this only required $K\ge 11$.)

It remains to determine the expansion properties   (in $\P_{n,d+1}$) implied when none of the events  $R_{i,\beta}$ hold, and additionally the initial number $x$ of unexposed points in the exposure process applied to $H(V',r_0-1)$ satisfies the lower bound in~\eqn{xbound}. We will define $g(r')$ inductively for $r'\ge r_0-1$, such that there are at least $i_{r'+1}\big(1- g(r' )\big)$ points remaining unexposed in vertices in $S(V',r )$ at the start of round $r'+1$. From~\eqn{xbound} and~\eqn{x0}, we satisfy  the first case of this requirement by setting $g(r_0-1)= (2+2\eps)/(d+1)$ if $r_0\ge 2$, and $g(0)=0$ when $r_0=1$. For any $r\ge r_0 $, consider round $r$. Each of the less than $\beta_ri_r$ bad edges either ``wastes'' two of the unexposed points in $S(V',r-1)$, or joins one such point to a point in $S(V',r)$. The first case reduces the eventual number   of unexposed points in $S(V',r)$ by $2d$ since two fewer new vertices are reached; in the latter case the reduction is less than this. Since $i_{r+1}=di_r$, we satisfy the general requirement for $g(r)$ by defining $g(r) = g(r-1)+2\beta_r$, and thus
$$
 g(r)= g(r_0-1)+\sum_{j= r_0}^r 2\beta_j.
$$

Note that at most $r=O(\log n)$ values of $j$ have $\beta_j = (1+\eps)(\log n)^{-2}$, and at most $O(\log \log n)$ values of $j$ have $i_j\ge n/\log^{10}n$, in which case $\beta_j=(\log \log n)^{-2}$. Thus
$$
\sum_{j= r_0}^r 2\beta_j\le  o(1)+  \sum_{j= r_0}^\infty 2(1+\eps)k/i_r =  o(1)+ \frac{2(1+\eps)}{(d+1)d^{r_0-2}(d-1)}
$$

First consider $d\ge 3$. Here $g(r) \le  g(r_0-1) + \frac14(1+\eps)/3^{r_0-2}$ which is at most $\frac34(1+\eps)$ in both cases $r_0=1$ and $r_0\ge 2$. Since the number of   points remaining unexposed in vertices in $S(V',r )$ at the end of round $r$ is at least $i_{r+1}\big(1- g(r )\big)$, and each vertex contains at least one exposed point, the number of vertices is at least $i_{r}\big(1- g(r )\big)\ge (1/4-\eps)k(d+1)d^{r-1}$. Taking $\eps < \eta$, we obtain the conclusion of (ii) for $d\ge 3$.

Now consider $d=2$. In this case, the above argument does  not   quite suffice, so we modify it. The weakness in the argument at present  is that we are permitting events with probability close to $1/{n \choose k}$ in every round. So consider running two rounds at a time and counting the bad pairs that occur within the two rounds. Thus, there are up to  $3i_r $   pairs exposed during rounds $r$ and $r+1$. So we flip this many coins and what we need is~\eqn{wanted} with $i_r$ replaced by $i_r'=3i_r$ (noting that the bound $\gamma = O(i_r/n)$ is still valid). Also set $\beta_r' =\beta_r/3$. Then $\beta_r'i_r'=\beta_ri_r$, $i_r'/\beta_r'=O(i_r/\beta_r)$ and so~\eqn{wanted} itself implies the   version of~\eqn{wanted}  with $i_r$ and $\beta_r$ replaced by $i_r'$ and $\beta_r'$. Noting that one bad pair eliminates at most $2d^2=8$ points from availability two rounds later, the number of points eliminated  is at most $8\beta_r' i_r'= 8\beta_r  i_r =2\beta_r  i_{r+2}$. So    we may put   $g(r+1)=g(r-1)+2\beta_r $ and skip  every second value of $r$ in the summations. In the   case $r_0=1$, we have  
\begin{eqnarray*}
g(r_0+2\ell+1)&=&  \sum_{j= 0}^\ell 2\beta_{r_0+2j}  \\
&\le& o(1)+  \frac{2+ 2\eps}{3}  + \frac{2+ 2\eps}{12}+\cdots \\
&\le& \frac{8(1+\eps)}{9}+o(1)
\end{eqnarray*}
as required (again taking $\eps<\eta$).  Of course  constant $8/9$ could be reduced down to any constant greater than $2/3$, by further tweaking the argument and looking more rounds ahead. 
 
For $r_0\ge 2$ we have analogously
\bel{firstg}
g(r_0+2\ell+1)= \frac{2+2\eps}{d+1}+ \sum_{j= 0}^\ell 2\beta_{r_0+2j}. 
\ee
This still does not suffice (even with the potential improvement mentioned above), because of the estimate of the number of points already wasted in rounds before $r_0$. The above argument works assuming only $K\ge 13$, and at this point we will use the fact that $K\ge 17$. First assume    $r_0\ge 3$, and define  $r_0^-=r_0-2$. Then $\beta_r= (2+2\eps)/(12d^{r-r_0})$ for small values of $r$, and  the same argument as above is still valid, since the left hand side of~\eqn{wanted} is now $(\wO(1))^k(k/n)^{\Omega(d^{-2}\log^K n)}= o(n^{-2}e^{-d^2})$ as required. Thus~\eqn{firstg} becomes
$$
g(r)\le o(1)+\frac{2+2\eps}{3}  (1+ 1/3)\le \frac{8(1+\eps)}{9}+o(1)
$$
as required.

We are left finally with the case $r_0=2$. Since $K=17$,  this was actually covered by the first version of the argument above, after reducing $K$ by 4 so that this case appears as $r_0=1$ (since $d\le \log^4 n$).  This completes the proof of part (ii) of the lemma. \proofend 

 
\smallskip
 
\noindent
{\em Proof of  Lemma~\ref{l:basics}(iii)}.
The proof has four ``phases.'' These refer to various parts of the exposure process. Let $v\in V(G)$, and for $i\ge 0$ let  $L_i$ for brevity  denote $S(v,i)$, the vertices at distance $i$ from $v$. 
\medskip

\noindent
{\em Phase 1: run the exposure process for $t'$ rounds}
\smallskip

\noindent
Consider running the exposure process in (i) with $V'=\{v\}$, extended until the first complete round in which the total number $n'$ of vertices that have been encountered is at least $ n/d^3\log^2 n.$ (Note that there is nothing to prove in part (iii) if the graph is disconnected. Almost all of the rest of our proof is written as if  $G$ were connected and in particular such a round is actually reached. We do not wish to condition on this event, as that would alter the probabilities of various events in the exploration process. In all our remaining discussion of the exploration process, we implicitly assume that  the process  has not already reached a conclusion by exposing  a component of $G$ prematurely. Inserting such qualifications and appropriate modifications explicitly at each point is routine but would complicate the exposition considerably, so we omit them for clarity.) Let $t'$ be the index of this round, so the set $L_{t'}$ of vertices at distance $t'$ from $v$ has just been revealed.  Note also for later use that, since successive neighbourhoods grow by a factor at most $d$ each time,  we have 
\bel{nsize}
 n/d^3\log^2 n\le n'\le n/d^2\log^2 n, \qquad d^{t'} \ge n/\wO(1).
\ee
The expected number of loops encountered so far is  $O(dn'/n)$, since there have been at most $dn'$ pairs exposed and each one (conditional on the history of the process) has probability $O(1/n)$ of creating a loop. (Note that the number of unmatched points remains asymptotic to $n(d+1)$ throughout.) Similarly the expected number of multiple edges is $O(d^2n'/n)$. Hence, letting $G_0$ be the graph induced by the pairs exposed by the process up to this point, we have that     $G_0$  is simple with probability $1-o(1)$, and the convergence in $o(1)$ is uniform over all $v$. (By this, we mean there exists $g(n)=o(1)$ such that the probability is at least $1-g(n)$ for all $v\in V$.)

We aim to show that with probability $1-o(n^{-3})$ either $G$ is disconnected or some other properties of $G_0$ hold; after this, we will use an argument similar to that in part (i) to deduce that the same claim is true conditional on $G_0$ being simple. We avoid mentioning  the disconnected case henceforth, as explained above. Let $U_0=\bigcup_{u\in S(v,r)} S(u,r')$. Our strategy for proving accessibility will   first be to find, with probability  $1-o(n^{-3})$, some large disjoint sets $W_0(w)\subseteq V(G_0)$ for almost all vertices in $U_0$. These sets will then be ``grown'' outside $G_0$ where necessary to form the larger sets $W(w)$, for each $w\in U$.  The probability they cannot be grown with their desired properties, for any particular $U\subseteq U_0$, will be so small that a union bound over all $U\subseteq U_0$ will yield the desired bound on the probability that the required sets $W(w)$ exist. (The sets $W_0(w)$ for $w\notin U$ will not be used; this means a  loss of $o(n)$ vertices from those available to build trees outside $G_0$, but this creates no problem for the argument.) 

\medskip

\noindent {\em Phase 2: re-examine the process to round $r+r'+t$}
\smallskip

Note that each vertex $w\in U_0$ was reached in the exposure process by the time of completion of round $r+r'$. We  ``rewind'' the process to the end of round $r+r'+t$, where $t=\lfloor \log_d(\log^{12} n)\rfloor$.   Recalling $d\le \log^4n$, we note for later use that $d^t\ge \log^8 n$. After $r+r'+t$ rounds, using~\eqn{treesize} there are $O(d^{r+r'+t}) =  O(n^{1/2+2\delta}  \log^{12} n )$ distinct vertices containing exposed points. The number of pairs exposed is at most $d$ times this, and the probability that a given pair exposed up to this point is bad (as defined in the proof of part (i) of this lemma) is $O(n^{-1/2+2\delta}  \log^{12} n)$. So, the expected number of bad pairs is $O(dn^{ 4\delta} \log^{24} n)$. Thus, employing the argument in the proof of (i), and in particular stochastically bounding the number of bad pairs by a binomial random variable, we deduce that, with probability $1-o(n^{-3})$,  the total number of bad pairs encountered in $r+r'+t$ rounds   in the process starting at $v$ is $O(dn^{ 4\delta} \log^{24} n)=o(n^{5\delta} /d\log^{12} n)$. Here we use a consequence of Chernoff's bound (see e.g.~\cite[Corollary~2.3]{JLR}) to achieve the sharp tail estimate: if $X$ is distributed as ${\rm Bin}(n,p)$, then
\bel{chern}
\Prob( |X-\E X| \ge \eps \E X) ) \le 2\exp \left( - \frac {\eps^2 \E X}{3} \right)  
\ee
for  $0 < \eps < 3/2$.  

For each vertex  $w\in U_0$, let $\tau =\tau(w)$ denote the distance from $v$ to $w$; i.e.\ $\tau$ is such that $w\in L_{\tau }$. Set $L_0(w) =\{w\}$, and inductively for each $i\ge 0$ denote the  set of vertices in  $L_{\tau+i+1}$ adjacent to vertices in $L_i(w)$ by $L_{i+1}(w)$. Then set $T_j(w)= \bigcup_{i=0}^ {j}   L_i(w)$. The exposure process has by now revealed (amongst other things), for each $w\in U_0$, the set $T_{t+1}(w)$ of vertices, with  $t$   defined as above. We now place  each such  $w$ into the set $Q$  if and only if any vertex in $T_{t}(w)$ is incident with a bad pair. Then, from the observation above bounding the total number of bad pairs,
\bel{Qsize}
|Q|=o(n^{5\delta} )
\ee
with probability $1-o(n^{-3})$. (An extra factor $O(d\log^{12} n)$ was allowed for, as a crude upper bound on the number of vertices in $T_{t}(w)$, and hence on the number of vertices $w$ that each bad pair can eliminate.) 

If $w\notin Q$, it follows that   $T_{t}(w)$ induces a ``regular'' tree, which we denote by $T(w)$, with each vertex adjacent to precisely $d$ in the next level, apart from the leaves. Moreover, $T(w_1)$ and $T(w_2)$ must be  vertex-disjoint for any two distinct vertices $w_1$ and $w_2$ both in the same level set $ L_i$.  (They are not always disjoint if they come from different levels, since for example it is possible that $w_2\in L_2(w_1)$.)  

Note that $ L_{r+r'}\subseteq U_0$, but that other vertices of $U_0$  are scattered at various distances from $v$. We first   consider  $\hatU_0:=    L_{r+r'}\setminus Q$.   We have from above that the trees $T(w): w\in \hatU_0$ are pairwise disjoint regular trees. These trees are all based at the same level, which simplifies the presentation of our analytic arguments. See Figure~\ref{f:shells}(a).
\begin{figure}[h]
\begin{center}
\includegraphics[width=15cm]{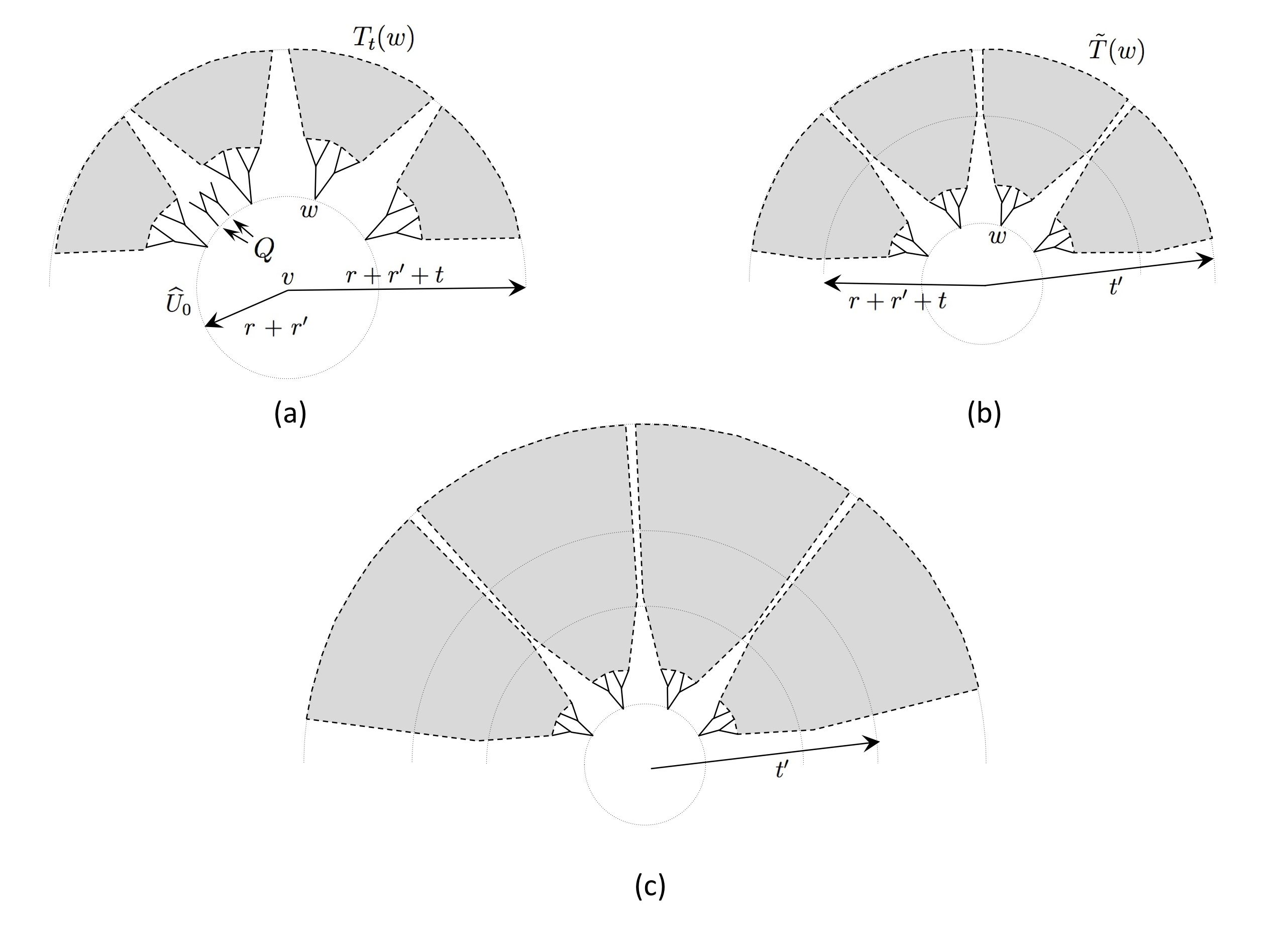}
\end{center}
\vspace{-5mm}
\caption{\it The trees growing from $\hatU_0$ are disjoint and (a) fully $d$-ary in Phase 2, (b) nearly $d$-ary in Phase 3, and (c) a constant fraction of the size of $d$-ary trees in Phase 4.
}
\label{f:shells}
\end{figure}

\medskip

\noindent{\em Phase 3: re-examine  from round $r+r'+t$ to round $t'$.}

\smallskip

\noindent We next re-examine the exposure process from level $r+r'+t$, and extend each induced tree $T(w)$ ($w\in\hatU_0$)  into a tree $\tilde T(w)$ that reaches ``up'' as far as vertices in  $L_{t'}$ (defined at the start of the proof of (iii), when $n'$ vertices are reached). This is done in a breadth-first manner, adding one level to all of the trees  before continuing to the next level. With each level, we expose all unexposed points in each vertex $u$ in  that level, and add each exposed edge  to the tree containing $u$ if the pair is not bad and if there is such a tree. The vertices which were reached only through vertices in $Q\cap L_{r+r'}$ do not belong to trees, but are still re-exposed.

At the time when all points in vertices in $L_{i-1}(w)$ have been exposed for all $w\in \hatU_0$, we say that stage $i$ is complete. At this point all vertices in $L_{i}(w)$ have been determined. This terminology is simply a re-indexed way to refer to round  $r+r'+i$. The trees $\tilde T(w)$ are fully grown when stage $t'-r-r'-1$ is complete. As a terminological note, we will sometimes refer to $L_{i}(w)$ as the set of vertices of  $\tilde T(w)$ at {\em height} $i$. The height of the tree is the maximum height of its vertices.
 
The next part of the argument will also apply to Phase 4. When a point $u$ is chosen for exposure and the pair $\{u, u'\}$ is  exposed, we say $u$ is the {\em primary} point of the pair and $u'$ the {\em secondary} point.  

Focus on what happens to just one of the trees,  $\tilde T(w)$, during stage $i+1$. Of the points in bad pairs exposed during this stage that are in vertices of $\tilde T(w)$ (whether at height $i$ or height $i+1$), let $X_i(w)$ denote the number of secondary points and $Y_i(w)$ the number of primary points. Suppose that $\Aconst$ is a deterministic upper bound on the total number of unexposed points in $\tilde T(w)$ at any time during the stage, and that there are deterministically at least $\xi n$ unexposed vertices at the end of the stage. Then, for each point exposed during this stage, the probability of creating a bad pair with a secondary point in $\tilde T(w)$ is at most $ \Aconst/\xi dn$ (here and in similar places, we could write $d+1$ instead of $d$ if we wished). If $\Bconst$ is a deterministic upper bound on the number of pairs exposed in total during this stage, we can stochastically bound $X_i(w)$  with a binomial random variable $\Bin( \Bconst ,  \Aconst/\xi dn)$, and apply Chernoff's bound~(\ref{chern}). 
This yields for any $\eps<3/2$
\bel{Xbound}
\pr\big(X_i(w)\ge (1+\eps)\Aconst \Bconst /\xi dn\big) \le 2e^{-\eps^2 \Aconst \Bconst /3\xi dn}.
\ee
Let $\Bconst '$ be an upper bound on the number of primary points to be  selected during this stage from vertices in $\tilde T(w)$, and $\Aconst '$ a deterministic upper bound on the total number of unexposed points that can lie in exposed vertices at any time during the stage. We will stop the process as soon as there are less than $\xi n $ unexposed vertices, and hence there are at least $\xi dn$ unexposed points within unexposed vertices.  Thus, the probability of a pair exposed during this stage being bad is at most $\Aconst '/\xi dn$, and we can stochastically bound $Y_i(w)$  with a binomial random variable $\Bin( \Bconst ', \Aconst '/\xi dn)$, yielding
\bel{Ybound}
\pr\big(Y_i(w)\ge (1+\eps)\Aconst '\Bconst '/\xi dn \big) \le 2e^{-\eps^2\Aconst '\Bconst '/3\xi dn}.
\ee
 
We return to focus on Phase 3.  We will show by induction on $i\ge t$ that with probability   $1-o(i/n^4)$, at the completion of stage $i$ each of the trees   $\tilde T(w)$ contains  at least a $(1-\eps_i)$ fraction of its maximum possible number $d^{i+1}$ of unexposed points in vertices in $L_i(w)$, where $\eps_i$ is to be defined below  but in any case will be at most $1/2$ (in fact, $\eps_i=o(1)$). From the way that $\hatU_0$ was defined so as to exclude $Q$, we know that $\tilde T(w)$ contains all $d^t$ vertices in $L_t(w)$, each containing $d$ unexposed points at the completion of stage $t$, so we may put $\eps_t=0$. 

For $i\ge t$, consider stage $i+1$. Here  $ d^{i+1}\ge d^{t+1}\ge  d\log^8n$ as noted above.  Deterministically, there are at most $\Aconst :=d^{i+2}$ points unexposed in $\tilde T(w)$ at any time. The number of points exposed in this stage is at most $ dn'\le  \Bconst := n/d\log^2 n$. So by~\eqn{Xbound} with $\xi=1/2$ and $\eps=1$,  the probability that  $X_i(w)\ge 4d^{i+2}/d^2\log^2 n$ is $\exp(-\Omega(d^i/\log^2 n ))= \exp(-\Omega(\log^6n)) = o(1/n^5)$. On the other hand, we can set $\Bconst '=d^{i+1}$ and $\Aconst '=  n/d\log^2 n$ to obtain a similar bound on the probability that   $Y_i(w)\ge   4d^{i+1}/d^2\log^2 n$ via~\eqn{Ybound}. The number of bad pairs generated in stage $i$ involving points in $\tilde T(w)$ is at most $X_i(w)+Y_i(w)$. This  bound therefore holds,  for all trees $\tilde T(w)$, with probability $1-o(n^{-4})$. All non-bad pairs  are added to the trees. By induction, there were at least $(1-\eps_{i})d^{i+1}$ points waiting to be exposed in $\tilde T(w)$ in stage $i+1$, so in this case the number of vertices  added to $\tilde T(w)$ at height $i+1$ is at least  
$$
(1-\eps_{i}) d^{i+1} - X_i(w) - Y_i(w)  \ge    (1-\eps_{i}- 6/d\log^2 n )d^{i+1}.
 $$
The number of unexposed points is at least $d$ times this quantity, minus $X_i(w)$. Hence with $\eps_i$ defined as say  $4 (i-t)/\log^2 n$,  the inductive claim is established.
 
Setting $i=\min\{t'-r-r', 2\log_d n\}$, we obtain $\eps_i=o(1)$, and it is easily seen to follow from this that $i=t'-r-r'$. Thus, with probability $1-o(n^{-3})$,  the trees $\tilde T(w) $  can be defined for all $w\in \hatU_0$, so that they have the following property. See Figure~\ref{f:shells}(b).

\begin{property}\lab{obs} The trees $\tilde T(w_1) $  and $\tilde T(w_2) $  are pairwise disjoint for $w_1,\, w_2\in \hatU_0$. Furthermore, each tree $\tilde T(w) $ contains   $d^{t'-r-r'} (1-o(1))$   vertices at the top level, i.e.\ in $L_{t'}$, and  $d^{t'-r-r'+1} (1-o(1))$ unexposed points in these vertices. 
\end{property}

This holds in the pairing model, but we can now argue as in the proof of (i), as explained above at the start of the proof of (iii), that the same observation holds with probability $1-o(1/n^3)$ conditional upon $G_0$ being a simple graph. Note that the re-run process   has re-exposed all vertices of $G_0$, but the trees would not contain many vertices of  $G_0$ that lie ``above'' vertices in $Q$.

\medskip
\noindent{\em Phase 4: the exposure process after round $t'$}
\smallskip

\noindent
 We can now condition on the so-far-exposed subgraph $G_0$ of $G$ satisfying the event shown in the observation.  Let  $A\subseteq S(v,r)$ and  $U=\bigcup_{u\in A} S(u,r')$ with $|A|> n^{1/4-\delta}$ and $ d^{r+r'}< n/\gammaconst|U|$. Since there are at most $2^{n^{1/4+\delta}}$ choices of the set $A$ and $n$ choices of $w$, we are done by the union bound once we  show that the probability that any particular one of the sets $A$, the required sets $W(w)$ exist with the desired properties with probability $1-o(2^{-n^{1/4+\delta}}/n^4)$. 
 
To introduce the argument in a simple  setting, we essentially treat the case  $U\setminus Q\subseteq \hatU_0$ first. We will next grow the trees $\tilde T(w) $ to height $r+r'+1$. See Figure~\ref{f:shells}(c). For each $w$, the set $W(w)$ will be chosen from the vertices of $\tilde T(w)$ at maximal distance from $w$. (For the later case when $U\setminus Q\not \subseteq \hatU_0$, the trees will be grown to different heights, and $W(w)$ may contain vertices of several trees  $\tilde T(w')$  where $w'\in \hatU_0$, at a level that lies within $N(w,r+r'+1)$ as required.)

\begin{sublemma}\lab{newclm}
Condition on  $G_0$ having  trees $\tilde T(w) $ for all $w\in \hatU_0$ with Property~\ref{obs}. Suppose that $U^*\subseteq \hatU_0$ with $|U^*|<   n/\gammaconst d^{r+r'}$. With probability  $1-  o(2^{-n^{1/4+\delta}}/n^4)$,  we can find a family of pairwise disjoint trees $\{\tilde T(w):w\in U^*\}$ in $G$ such that for  $w\in U^*$ and $i\le r+r'+1$ the number of vertices in $\tilde T(w)$ at height $i$  is at least
$$
\frac25\min\left\{ d^{i},\frac{n}{\accessconst |U^*|}\right\}.
$$
\end{sublemma}
\begin{proof}
Conditioning on the subgraph $G_0$ leaves the remaining part of $G$ to be a graph in which the vertices in $L_{t'}$ have given degrees, at most $d$, those in $L_i$ where $i<t'$ have degree 0, and all others have degree $d+1$. Again we consider the pairing model for this graph, which can be regarded as a subset of $\P_{n,d+1}$ in the obvious manner.  

For simplicity and later reference, let $\chi$ denote $|U^*|$. We will continue by generating extensions of all the trees $\tilde T(w) $ for every $w\in U^*$, up to $L_{r+r'}(w)$, and then continue with the next level as well, or as much of it is required. In all cases, we will stop the process before the number of as-yet-unexposed vertices drops below $n/2$.  Since $|L_i(w)| \le d^{i}$ and $d\ge 2$, the total number of vertices added to all the trees from  stage 1 to the end of stage $r+r'$ is at most $2d^{r+r'}\chi < 2n/\gammaconst$ by the hypothesis of the sub-lemma. Additionally, the number of vertices already found in the first $r+r'$ steps of the exploration process from $v$ is $o(n)$.  Thus, the process continues at least  to the end of stage $r+r'$. In case stage $r+r'+1$ is only partially used, we will desire the numbers of vertices at height $r+r'+1$ in different trees to be roughly equal (with high probability). This is  achieved by exposing the same number of points in each tree. Note that in this phase of the exploration, we {\em only} explore the vertices in the trees $\tilde T(w)$ and their extensions, and do not start explorations from the other vertices of $G_0$. 

Define $i_0:=t'-r-r'$. For $w\in U^*$,  let  $n_{i_0}(w)$ denote the number of  vertices in the tree $\tilde T(w)$  at the top level (before generating extensions), i.e.\ at height $i_0$. From Property~\ref{obs},  $n_{i_0}(w) = (1+o(1)) d^{i_0}= (1+o(1)) d^{t'-r-r'}$, and the number of points in vertices at this level of the tree, say $N_{i_0}(w)$, is asymptotic to $ dn_{i_0}(w)$. Hence, using~\eqn{nsize} and the bounds on $r$ and $r'$ in the hypotheses of~(iii), together with the fact that $(d+1)d^{t'-1}=\Omega(n')$ (by definition of $n'$), we have
\bel{somebounds}
d^{i_0+1} = (1+o(1)) N_{i_0}(w)> n^{1/2-2\delta}/\wO(1).
\ee

Now consider stage $i+1$ where  $i\ge i_0 $, fix  $\gamma=\gammaconst$, and assume that  
\bel{icond}
 d^{i+1}< n/\gamma\chi.
\ee  
Recall that $d^{i+1}$ is an upper bound  on the number of pairs in $\tilde T(w)$ exposed in stage $i+1$, and hence also on the number of  vertices in $\tilde T(w)$ at height $i+1$. Hence, at the end of this stage,   the total number of exposed vertices is at most $o(n)$ (for vertices exposed by stage $i_0$) plus  $2n/\gamma$ (summing the geometric series using $d\ge 2$ and multiplying by the number of trees,  $\chi$). Thus we can take any fixed $\xi <1-2/\gamma$   in~\eqn{Xbound}, with fixed $1>\eps>0$, $\Aconst =d^{i+2}$ and $\Bconst =   d^{i+1}\chi'$ where $\chi'=\max\{n^{1/2-2\delta-\eps},\chi\}$, giving  
$$
\pr\big(X_i(w)\ge (1+\eps)d^{2i+2}\chi'/\xi n\big) \le 2\exp(- \eps^2d^{2i+2}\chi'/3\xi n).
$$
From~\eqn{somebounds}, $d^{i} >d^{i_0} >n^{1/2-2\delta}/\wO(1)$, and so  we obtain  
$$
\pr\big(X_i(w)\ge (1+\eps)d^{2i+2}\chi'/\xi n\big) = o\big(\exp(- n^{1/2-6\delta- 2 \eps}) \big) = o\big(\exp(- n^{1/4+2\delta})\big) 
$$
for (as we may assume) $2 \eps< \frac{1}{4}-8\delta$. (Actually, there are fewer secondary points expected at height $i$ than height $i+1$, which we could take advantage of to slightly improve the constants in this proof and/or the definitions. Ultimately, when we apply this lemma, such an improvement  would only reduce the number of cops required by a constant factor, at the expense of slightly complicating the argument.) 

To apply~\eqn{Ybound} we may once again set  $\Bconst '=d^{i+1}$, and this time $\Aconst '= d^{i+2}\chi'$, which gives exactly the same bounds on $Y_i(w)$ as $X_i(w)$. Since the ``failure'' probability in the sub-lemma statement is $o(2^{-n^{1/4+\delta}}/n^4)$, we may now assume that both $X_i(w)$ and $Y_i(w)$ are less than $ (1+\eps)d^{2i+2}\chi'/\xi  n$, for all $i<\log_d n$ say, as well as all   $w\in U^*$, by applying the union bound.

In general let $n_i(w)$ denote    the number of vertices in $\tilde T(w)$ at height $i$, and $N_i(w)$ the number of unexposed points in these vertices at the end of stage $i$. Then $n_{i+1}(w)$  is at least $ N_{i }(w) -X_i(w)-Y_i(w)$, since this is a lower bound on the number of points exposed in good pairs in stage $i+1$. Thus, we may assume that 
\bel{nbounds}
N_{i}(w) - \delta_i \le n_{i+1}(w)\le N_{i }(w) 
\ee
where $\delta_i = 2(1+\eps)  d^{2i+2} \chi'/\xi n$. Similarly, the number of unexposed points at the end of stage $i+1$ would be $dN_i(w)$ if no bad pair occurred, and this is reduced by at most $d$ for each point counted by $X_i(w)$ or $Y_i(w)$, so 
$$
d\big(N_{i}(w)  - \delta_i\big)\le N_{i+1}(w)\le dN_{i }(w).
$$
From the upper bound  $\chi < d^{-r-r'}n/\gamma$ assumed in the sub-lemma,~\eqn{icond} holds for $i\le r+r'-1$, and hence   these conclusions apply for such $i$. Iterating, we obtain
$$
N_{r+r'}(w)\ge d^{r+r'-i_0}N_{i_0}(w)-\sum_{j=1}^{r+r'} d^{j}\delta_{r+r'-j}.
$$
The assumptions in the lemma, that $(d+1)d^r<n^{1/4+\delta}$ and similarly for $r'$, imply that  $n^{1/2-2\delta-\eps} < d^{-r-r'}n/\gamma$ ($n$ large), and thus  $\chi' < d^{-r-r'}n/\gamma$  as for $\chi$. Hence  we may assume 
$$
\delta_i< 2(1+\eps)d^{2i+2-r-r'}/\xi\gamma
$$
and consequently
\begin{eqnarray*}
N_{r+r'-1}(w)&\ge&  d^{r+r'-i_0}N_{i_0}(w)-d^{r+r'+1}\sum_{j=0}^{\infty}2(1+\eps)d^{-j}/\xi\gamma\\
&\ge&  d^{r+r'+1}\left(d^{-i_0-1}N_{i_0}(w)- 4(1+\eps)/\xi\gamma\right).
\end{eqnarray*}
Since this is valid for any $\xi <1-2/\gamma$ and sufficiently small $\eps>0$, we see using~\eqn{somebounds}  that for $n$ sufficiently large
$$
N_{i}(w)>\alpha d^{i+1}
$$
for any $i\le r+r'$ and 
any fixed $\alpha$ satisfying
$$
\alpha < (\gamma-6)/(\gamma-2).
$$
A similar argument but using~\eqn{nbounds}  for the last step yields 
\bel{finaln}
n_{i}(w)>\alpha d^{i}\quad(\mbox{all }i\le r+r').
\ee
Since we  set $\gamma=\gammaconst$, we may put $\alpha=2/5<3/7$, which gives the conclusion of the sub-lemma for all $i\le r+r'$.

For $i=r+r'+1$, consider two cases. Firstly, if $d^{r+r'+1}< n/\gamma \chi$, then the above argument applies equally well to one more level of the trees, i.e.\ reaching height $r+r'+1$, and we have the analogue of~\eqn{finaln}: $n_{r+r'+1}(w)\ge  \alpha d^{r+r'+1}$, as required.
  
On the other hand, if  $d^{r+r'+1 }\ge n/\gamma \chi$, then we can instead build to height $r+r'+1$ only partially, building the trees in a balanced fashion as described above, exposing precisely $\lfloor \beta n/\chi\rfloor$ points at height $r+r'$ of each tree for some $\beta>0$ to be chosen below. Letting
$$
\xi = 1-\beta-1/\gamma-\eps
$$
we can set $a=dn(1-\xi)/\chi$, $b=\beta n$, $b'=\beta n/\chi$, $a'=dn(1-\xi)$ for use in~\eqn{Xbound} and~\eqn{Ybound}. Here $ab/dn=\Omega(n^{1/2-2\delta})$, and by the argument as before  we may assume that 
$$
X_{r+r'}(w)+Y_{r+r'}(w) <\frac{ 2(1+\eps)n\beta(1-\xi)}{\chi\xi}.
$$
Thus, asymptotically at least $n\rho/\chi$ vertices are added to the next layer of each tree, for any 
$$
\rho <  \beta\left(1-\frac {2(\beta+1/\gamma)}{1-\beta-1/\gamma}\right).
$$
Setting  $ \beta = 1/\gammaconst$ and recalling $\gamma=\gammaconst$, this condition says $\rho< (3/7)(1/9)$, so we can take $\rho = (2/5)(1/9)$ and again we have the sub-lemma.
 \end{proof}

First suppose that $U^*$ happens to equal $U\setminus Q$. We have $|U^*| \le |U|<  n/\gammaconst d^{r+r'}$ by the hypothesis of~(iii). Hence, the sub-lemma implies the existence of trees $\tilde T(w)$ for each $w\in  U\setminus Q$. For each $w\in U^*$, the set $W(w)$ can now be chosen from vertices of the tree $\tilde T(w) $ at height $r+r'+1$. The lemma follows in this case.
  
To complete the proof of the lemma, we   need to cope with the fact that some  vertices of $U$ can lie at distances less than   $r+r'$ from $v$, i.e.\ in $L_{r+r'-j}$ for some $j>0$.  To do this we will need need to note that the sub-lemma can easily be extended to a version in which   the trees are not all of the same height.

For $j\ge 0$ let $R_j= L_{r+r'-j}\cap U\setminus Q$. Firstly, the argument leading up to Property~\ref{obs} is easily adapted to show that for given $j$, with probability at least   $1-n^{-4}$, we can grow disjoint trees $T(w)$ from all vertices $w\in R_j$ level by level up to level $ {r+r'}$ so that each contains asymptotically $d^{j}$ vertices in  $L_{r+r'}$, and contains no vertex in $Q$. Let $F(w)$ denote the set of vertices in $V(T(w))\cap L_{r+r'}$.
  
Let $\tildeU$ denote the union of the sets $F(w)$ over all $w\in U$. Let   $\tildeU_j$ be the subset of $\tildeU$ containing those vertices $w$ for which   $j$ is the minimum value such that $w\in F(w')$ for some $w'\in R_j \subseteq L_{r+r'-j}$. We can condition on  $G_0$ having  trees $\tilde T(w) $ for all $w\in \hatU_0$ with Property~\ref{obs}. Then these are disjoint trees based on all vertices in $\tildeU$, each of  height $t'-r-r'$. 
 
We will show the existence (with the required probability) of disjoint extensions of the trees $\tilde T(w)$ for each $w\in \tildeU$ such that for $w\in \tildeU_j$ the height of $\tilde T(w)$  is  $r+r'+1-j$, and the number of vertices at height $i$  is at least 
\bel{newlb}
 \frac25\min\left\{  d^{i},\frac{nd^{-j}}{\accessconst |U|}\right\}
 \ee
for all $1\le i\le r+r'+1-j$. So actually if $w\in \tildeU_j$ for  $r+r'+1-j\le t'-r-r'$, the tree $\tilde T(w)$ does not need to be grown any further than its present height. Hence, we may assume that $r+r'+1-j>  t'-r-r'$, i.e.\ $j\le j_{\max}:= 2r+2r'-t'$.

The existence of these trees lets us find the required sets $W(w)$ treating the $w$'s level by level. For those $w\in R_{0}$ we use $W(w)$ as defined in the simple case discussed above. In general,  for those $w\in R_j$, $W(w)$ is the set of all vertices at height $r+r'+1-j$ in all the trees $\tilde T(v')$ for $v'\in F(w)$. Since the trees have disjoint level sets, $W(w_1)$ and $W(w_2)$ are disjoint whenever $w_1$ and $w_2$ are not in the same set $R_j$, whilst if they are in the same $R_j$, we have $F(w_1)\cap F(w_2)=\emptyset$ and thus $W(w_1)$ and $W(w_2)$ are again disjoint. Of course, the size of $W(w)$ is then, given the bound~\eqn{newlb}, at least  
$$
|F(w)| \frac25 d^{-j}\min\left\{  d^{r+r'+1 },\frac{n}{\accessconst |U|}\right\}.
$$

Since $|F(w)| \sim d^j$, this proves  part (iii) of the lemma provided that the required extensions of   the trees $\tilde T(w)$ exist. To show these exist, we use a quite simple adaptation of the proof of Sub-lemma~\ref{newclm}.   For  $w\in \tildeU_j$ the present height of $\tilde T(w)$ is $t'-r-r'$, and we begin by pruning off the top $j$ layers of this tree. The conditioning we assume is now on $G_0$ with all these pruned vertices removed. After all the prunings, each tree needs to grow an extra $r+r'+1-(t'-r-r')$ layers, just as for the proof of Sub-lemma~\ref{newclm}. For this new version of phase 4, we re-define stage $i$ to refer to the exposure of all points in vertices at height $i-1-j$ in   trees $\tilde T(w)$ with $w \in \tilde U_j$, for any  $j\le j_{\max}$. Thus, stage $t'+1$ adds one more layer to all the current trees. Note that we now re-expose the parts of $G_0$ that were pruned, but this part of the graph is not being conditioned upon so the remaining part of the graph can be regarded as random.  

The main difference from the situation in the sub-lemma is that the trees are not now all of the same   sizes (approximately) at a given level. Now, for $w\in \tildeU_j$, the number of vertices of $\tilde T(w) $ at height   $i_0 =t'-r-r'$ is asymptotic to $d^{i_0-j}$. So, for the present context, we let $\chi= \sum_{j=0}^{j_{\max}}d^{-j}|\tildeU_j|$, and hence the number of vertices currently (at the end of stage $t'$) at the maximum height, in total   among all the trees, is asymptotic to  $d^{i_0}\chi$.  Since $j\le j_{\max}$, we have  $d^{i_0-j}>n^{1/2-6\delta}/\wO(1)$, similar to~\eqn{somebounds}. Thus, assuming~\eqn{icond}, basically the proof of Sub-lemma~\ref{newclm} still applies to show that~\eqn{nbounds} can still be assumed, with $\chi'$ defined using the same formula  as before but using the  new definition of $\chi$, and $\delta_i = 2(1+\eps)  d^{2i+2-j} \chi'/\xi n$. Eventually we have $n_{i-j}(w)>\alpha d^{i-j}$ for all $i\le r+r'$ in place of~\eqn{finaln} (with all but negligible probability).

As in the proof of Sub-lemma~\ref{newclm}, for the case $d^{r+r'+1 }< n/\gamma \chi$ we obtain~\eqn{newlb} for the top level of the trees, as required. For the other case, as before we can use a partial extension of all the trees, the size of the extension to each tree relating to $ \tildeU_j$ being  proportional to $d^{-j}$, and obtain the same result.
\end{proof}


\begin{thebibliography}{99}
\bibitem{af} M.\ Aigner and M.\ Fromme, A game of cops and robbers, \emph{Discrete Applied Mathematics} \textbf{8} (1984) 1--12.

\bibitem{AS} N.~Alon and J.H.~Spencer, The Probabilistic Method, Wiley, 1992 (Fourth Edition, 2016).

\bibitem{al} B.\ Alspach, Sweeping and searching in graphs: a brief survey, \emph{Matematiche} \textbf{59} (2006) 5--37.

\bibitem{BS} S. Ben-Shimon and M. Krivelevich, Random regular graphs of non-constant degree: Concentration of the chromatic number, {\em Discrete Mathematics} {\bf 309}  (2009), 4149--4161.  

\bibitem{BB} B. Bollob{\'a}s, {\em Random Graphs. Second Edition}, Cambridge University Press, Cambridge, 2001. 


\bibitem{bn} A.\ Bonato and R.\ Nowakowski, {\em The Game of Cops and Robbers on Graphs}, American Mathematical Society, 2011.

\bibitem{book_BP} A.\ Bonato and P.\ Pra\l{}at, {\em Graph Searching Games and Probabilistic Methods}, CRC Press, 2017.



\bibitem{eshan} E.\ Chiniforooshan, A better bound for the cop number of general graphs, \emph{Journal of Graph Theory} \textbf{58} (2008), 45--48.

\bibitem{CFR} C.~Cooper, A.~Frieze, and B.~Reed, Random regular graphs of non-constant degree: connectivity and Hamilton cycles, {\sl Combinatorics, Probability and Computing} \textbf{11} (2002), 249--262.

\bibitem{Sandwich} A.~Dudek, A.~Frieze, A.~Ruci\'nski, M.~{\v S}iliekis,  Embedding the Erd{\H o}s-R{\' e}nyi  hypergraph into the random regular hypergraph and hamiltonicity, \emph{Journal of Combinatorial Theory, Ser.\ B} \textbf{122} (2017), 719--740.

\bibitem{f} P.\ Frankl, Cops and robbers in graphs with large girth and Cayley graphs, \emph{Discrete Applied Mathematics} \textbf{17} (1987), 301--305.

\bibitem{ft} F.V.\ Fomin and D.\ Thilikos, An annotated bibliography on guaranteed graph searching, \emph{Theoretical Computer Science} \textbf{399} (2008), 236--245.

\bibitem{fkl} A.~Frieze, M.~Krivelevich, and P.~Loh, Variations on cops and robbers, \emph{Journal of Graph Theory} \textbf{69} (2012), 383--402.

\bibitem{h} G.\ Hahn, Cops, robbers and graphs, \emph{Tatra Mountain Mathematical Publications} \textbf{36} (2007), 163--176.

\bibitem{JLR}  S.~Janson, T.~{\L}uczak, and A.~Ruci\'nski, {\em Random graphs}, Wiley, New York, 2000.     
  
\bibitem{KimVu} J.H.\ Kim and V.H.\ Vu, Sandwiching random graphs: universality between random graph models, \emph{Adv.\ Math.}, \textbf{188(2)} (2004), 444--469.

\bibitem{KSVW} M.~Krivelevich, B.~Sudakov, V.H.~Vu, and N.C.~Wormald, Random regular graphs of high degree, {\em Random Structures and Algorithms} \textbf{18} (2001), 346--363. 
  
\bibitem{lp} L.~Lu and X.~Peng, On Meyniel's conjecture of the cop number, \emph{Journal of Graph Theory} \textbf{71} (2012), 192--205.

\bibitem{lp2} T.~\L{}uczak and P.~Pra\l{}at, Chasing robbers on random graphs: zigzag theorem, \emph{Random Structures and Algorithms} \textbf{37} (2010), 516--524. 

\bibitem{M}  B. D. McKay, Asymptotics for symmetric 0-1 matrices with prescribed row sums, {\em Ars Combinatoria} {\bf 19A} (1985), 15--25.

\bibitem{nw} R.\ Nowakowski, P.\ Winkler, Vertex to vertex pursuit in a graph, \emph{Discrete Mathematics} \textbf{43} (1983), 230--239.


\bibitem{PW_gnp} P.\ Pra\l{}at and N.C.~Wormald, Meyniel's conjecture holds for random graphs, \emph{Random Structures and Algorithms} \textbf{48}(2) (2016), 396--421.

\bibitem{q} A.\ Quilliot, Jeux et pointes fixes sur les graphes, Ph.D.\ Dissertation, Universit\'{e} de Paris VI, 1978.


\bibitem{ss} A.~Scott and B.~Sudakov, A bound for the cops and robbers problem, \emph{SIAM J.\ of Discrete Math} \textbf{25} (2011), 1438--1442.

\bibitem{models} N.C.~Wormald, Models of random regular graphs, {\it Surveys in Combinatorics, 1999}, London Mathematical Society Lecture Note Series {\bf 267} (J.D. Lamb and D.A. Preece, eds) Cambridge University Press,  Cambridge, pp. 239--298, 1999.

\end{thebibliography}
\end{document}